\documentclass[a4paper,12pt,twoside]{amsart}

\usepackage{geometry}
\usepackage{amsthm}
\usepackage{amssymb}
\usepackage{amsmath}
\usepackage{enumitem}
\usepackage[british]{babel}
\usepackage[mathscr]{euscript}

\newtheoremstyle{thm}
{10pt} 
{10pt} 
{} 
{} 
{\bfseries} 
{.} 
{.5em} 
{} 

\newtheorem{tm}{Theorem}[section]

\newtheorem{lem}[tm]{Lemma}
\newtheorem{prop}[tm]{Proposition}
\newtheorem{sta}[tm]{Statement}

\makeatletter\@namedef{subjclassname@2020}{\textup{2020} Mathematics Subject Classification}\makeatother

\begin{document}

\newgeometry{textheight = 22.5 cm, textwidth=13.5 cm, centering}

\newcommand{\kc}{$\Sigma$-compact}
\newcommand{\kcd}{$\Sigma$-\textit{compact}}

\title[A generalization of the Open Mapping Theorem]{A generalization of the Open Mapping Theorem and a possible generalization of the Baire Category Theorem}
\author{Antoni Machowski}

\begin{abstract}
We characterize continuum as the smallest cardinality of a family of compact sets needed to cover a locally compact group for which the Open Mapping Theorem does not hold.
\end{abstract}
\subjclass[2020]{Primary 22D05; Secondary 03E17.}
\keywords{Baire Category Theorem; Lie group; locally compact; Open Mapping Theorem; sigma-compact; topological group.}
\maketitle
\author
\begin{center}
Institute of Mathematics, Jagiellonian University, Stanis\l{}awa \L{}ojasiewicza 6, 30-348 Krak\'ow, Poland
\end{center}

\section{Outline of the problem}

The Baire Category Theorem has multiple applications not only in topology but also in other branches of mathematics. This gives motivation to try to generalize it in multiple ways, one of which would be to consider it for uncountable families of nowhere dense sets. Such generalizations could provide some applications in the theory of cardinal numbers. We denote $\mathfrak{c}:=2^{\aleph_0}$. Consider the following statement.

\begin{sta}
\label{tm:Bairec}
Let $X$ be a Hausdorff topological space that is locally compact or completely metrizable. Then $X$ cannot be covered by less than $\mathfrak{c}$ nowhere dense sets.
\end{sta}

Unfortunately, this statement is undecidable in ZFC. \cite{Hec73} even shows that for any given regular non-denumerable
cardinal $\kappa$ there is a model of ZFC where $\kappa$ is the minimal number of nowhere dense sets needed to cover $\mathbb{R}$.  However, some special cases of Statement \ref{tm:Bairec} can be proved or disproved in ZFC. For example, we know that a nondiscrete uncountable topological group that is locally compact or completely metrizable cannot have cardinality less than $\mathfrak{c}$.

We now recall the well-known Open Mapping Theorem for topological groups. We will provide its classical proof to ilustrate the role that the Baire Category Theorem plays in it.

\begin{tm}
\label{Twr0}
Let $G$ and $H$ be locally compact groups and let $G$ be $\sigma$-compact. Then every continuous surjective group homomorphism $\varphi\colon G\to H$ is open.
\end{tm}
\begin{proof}
Let $U$ be an identity neighbourhood in $G$. We pick an identity neighbourhood $C$ in $G$ such that $C=C^{-1}$ and $C\cdot C\subseteq U$. Then we have $G=\underset{n\in\mathbb{N}}{\bigcup}g_n C$ for some $g_n\in G$ for $n\in \mathbb{N}$ and $H=\underset{n\in\mathbb{N}}{\bigcup}\varphi(g_n) \varphi(C)$. By the Baire Category Theorem $\varphi(C)$ has nonempty interior. Let $b\in \text{int}\varphi(C)$. Then $e=bb^{-1}\in \text{int}\varphi(C)\cdot \varphi(C) \subseteq \varphi(C)\cdot \varphi(C) = \varphi(C\cdot C) \subseteq \varphi(U)$, so $\varphi(U)$ is a neighbourhood of $e$ which suffices to finish the proof.
\end{proof}

We say that a topological group $G$ is \kcd\ if it can be covered by fewer than $\mathfrak{c}$ compact subsets. It can be readily seen that a closed subgroup of a \kc\ group is \kc\ and that a quotient of a \kc\ group by a closed normal subgroup is \kc. The goal of this paper is to prove the following statement.

\begin{tm}
\label{Twr1}
Let $G$ and $H$ be locally compact groups and let $G$ be \kc. Then every continuous surjective group homomorphism $\varphi\colon G\to H$ is open.
\end{tm}

Theorems \ref{Twr0} and \ref{Twr1} are obviously equivalent under CH. Our goal will be to prove Theorem \ref{Twr1} under ZFC and thus to provide a generalization of Theorem \ref{Twr0}. Our motivation for this generalization is that it gives a reason to suppose that some result closely related to Statement \ref{tm:Bairec} is true. As we will see in the following sections, the proof of Theorem \ref{Twr1} will use advanced structural theorems in the theory of locally compact groups, so it is not possible to repeat it for other classes of topological groups. It is not known to us whether Theorem \ref{Twr1} holds if we replace ``locally compact'' with ``completely metrizable''. We know however that the \kc\ condition cannot be replaced with a weaker condition admitting families of cardinality $\mathfrak{c}$ of compact sets since the identity map from $\mathbb{R}$ with discrete topology to $\mathbb{R}$ with euclidian topology is not open.

Our goal from now on will be to prove Theorem \ref{Twr1}. We start our efforts by recalling the following classical theorem. We skip its elementary proof.
\begin{prop}
\label{prop:Kr0}
Let $G$ and $H$ be topological groups and let $\varphi\colon G\to H$ be a continuous group epimorphism. Then
\begin{enumerate}[label=(\roman*)]
\item There exists exactly one continuous bijective group homomorphism $\widetilde{\varphi}\colon {G}/{\ker\varphi}\to H$ such that $\widetilde{\varphi} \circ \pi = \varphi$ where $\pi\colon G\to{G}/{\ker\varphi}$ is the quotient map.
\item For $\widetilde{\varphi}$ such as above, $\widetilde{\varphi}$ is open (equivalently, an isomorphism of topological groups) if and only if $\varphi$ is open.
\end{enumerate}
\end{prop}

We observe as a corollary to the above that it is enough to prove Theorem \ref{Twr1} in the case where $\varphi$ is bijective.

\section{Preliminaries}

In this section we introduce notions and auxiliary facts that will be used throughout the rest of the paper. By a group we mean a Hausdorff topological group. We denote by $G_0$ the identity component of a group $G$. It is a closed normal subgroup of $G$. We say after \cite{HM07} that $G$ is \textit{almost connected} if ${G}/{G_0}$ is compact. We recall the following simple fact.

\begin{lem}
\label{lem:ObsA}
Every almost connected locally compact group is $\sigma$-compact.
\end{lem}

The next fact is a straightforward corollary of the Van Dantzig Theorem (\cite[Theorem 1.6.7]{Tao14}).

\begin{prop}
\label{prop:VD}
Let $G$ be a locally compact group. Then there exists an open subgroup $\Omega$ of $G$ such that $\Omega$ is almost connected.
\end{prop}

We say that a topological group $G$ is a \textit{Lie group} if $G$ as a topological space is a smooth manifold. The group $G$ \textit{has no small subgroups} (or is an \textit{NSS group} for short) if there exists an identity neighbourhood $U$ which contains no nontrivial subgroups of $G$. We now invoke a theorem that links these two notions. It was a crucial result in resolving Hilbert's fifth problem and it is very important in our considerations since we will use both the simplicity of the NSS notion and some of the machinery of differential geometry (i.e. the one-parameter subgroups provided by the exponential mapping).

\begin{tm}[\cite{Yam53}, see also {\cite[Chapter 1]{Tao14}}]
\label{tm:NSS}
Let $G$ be a locally compact group. Then $G$ is a Lie group if and only if it has no small subgroups.
\end{tm}

The following result is the celebrated Gleason--Yamabe Theorem on approximating locally compact almost connected groups by Lie groups.

\begin{tm}[\cite{Gl51}, \cite{Yam53}, see also {\cite[Theorem 1.1.17]{Tao14}}]
\label{tm:GY}
Let $G$ be a locally compact almost connected group. Then for any identity neighbourhood $U$ in $G$ there exists a compact normal subgroup $K$ of $G$ such that $K\subseteq U$ and ${G}/{K}$ is a Lie group.
\end{tm}

We finish this section by recalling a well-known structual theorem (cf. e.g. {\cite[Theorem 9.24(iii)]{HM06}}).
 
\begin{tm}
\label{tm:Po}
Let $S$ be a compact connected group. Then there is a compact central totally disconnected subgroup $Z$ of $S$ such that ${S}/{Z}$ is a (possibly infinite) product of compact connected Lie groups.
\end{tm}

\section{Open Mapping Theorem with Lie target group}

The goal of this section is to prove Theorem \ref{Twr1} with the additional assumption that $H$ is a Lie group. The following auxiliary fact is surely well-known, but we could not find it in the literature. Thus, for the reader's convenience, we present its proof.

\begin{prop}
\label{prop:Aneks}
Let $G$ and $H$ be connected Lie groups and let $\varphi:G\to H$ be a continuous group homomorphism such that $[H:\varphi(G)]<\mathfrak{c}$. Then $\varphi$ is surjective.
\begin{proof}
Assume that $\varphi$ is not surjective. Then there exists a one-parameter subgroup, i.e. a continuous group homomorphism $\tau:\mathbb{R}\to H$ such that
\begin{equation}
\label{tau}
\tau(\mathbb{R})\nsubseteq \varphi(G).
\end{equation}
Let $S=\{(g,t)\in G\times\mathbb{R}:\varphi(g)=\tau(t)\}$. Then $S$ is $\sigma$-compact and in particular $S$ has countably many connected components. Now we observe that $S_0\subseteq G\times \{0\}$, because for the canonical projection $\pi_2:G\times \mathbb{R}\to \mathbb{R}$, $\pi_2(S_0)$ is a connected group, so it can be either $\{0\}$ or $\mathbb{R}$ and the second case is impossible by \eqref{tau}. Since $S$ has countably many connected components, $[S:S_0]$ is countable and so is $\pi_2(S)$. We observe that $\pi_2(S)=\tau^{-1}(\varphi(G))$ and we define $L=\pi_2(S)$. Pick $A\subseteq \mathbb{R}$ that meets each coset of $L$ at a single point. Then $\textbf{card}(A)=\mathfrak{c}$ and the mapping $A\ni a \mapsto \tau(a)\in H$ is one-to-one which means that $\textbf{card}[H:\varphi(G)]=\mathfrak{c}$ and leads to a contradiction.
\end{proof}
\end{prop}

Before we prove this section's main result, we invoke the following fact.

\begin{prop}
\label{prop:Obs1/2}
Let $G$ be a \kc\ group and let $H$ be an open subgroup of $G$. Then $[G:H]$ is smaller than $\mathfrak{c}$.
\begin{proof}
Let $G=\underset{j \in J}{\bigcup}C_j$ where $C_j$ are compact for $j\in J$ and $\textbf{card} (J) < \mathfrak{c}$ and let $C_j \subseteq \overset{n_j}{\underset{i = 1}{\bigcup}}g_i^{(j)}H$. Then we have $G\subseteq \underset{(i,j)\in \widetilde{J}}{\bigcup}g_i^{(j)}H$ for $\widetilde{J}=\{(i,j):j\in J, 1\leq i \leq n_j\}$ and $\textbf{card}(\widetilde{J})<\mathfrak{c}$.
\end{proof}
\end{prop}

We are now ready to prove the Open Mapping Theorem with Lie target group.

\begin{lem}
\label{lem:Kr1}
Let $G$ be a locally compact \kc\ group and let $H$ be a Lie group. Then every continuous bijective group homomorphism $\varphi\colon G \to H$ is open.
\begin{proof}
We observe that $G$ is an NSS group. Then $G_0$ is a connected Lie group and by Lemma \ref{lem:ObsA} it is $\sigma$-compact. By Propositions \ref{prop:Aneks} and \ref{prop:Obs1/2} we get that $\varphi(G_0)=H_0$. By Theorem \ref{Twr0}, $\varphi|_{G_0}$ is open as a map from $G_0$ to $H_0$, but both $G_0$ and $H_0$ are open, so $\varphi$ is open.
\end{proof}
\end{lem}

\section{Open Mapping Theorem with compact target group}

In this section we will prove Theorem \ref{Twr1} with the additional assumption that $H$ is a compact group. The following result is well-known, but we present its short proof for reader's convenience.

\begin{prop}
\label{prop:ObsF}
Let $X$ be a compact Hausdorff space of cardinality less than $\mathfrak{c}$ such that for every $x,y\in X$ there is a homeomorphism $f\colon X\to X$ satisfying $f(x)=y$. Then $X$ is finite.
\begin{proof}
If $X$ has an isolated point, then every point of $X$ is isolated, so $X$ is discrete and thus finite. If $X$ has no isolated points we can construct a subset of $X$ of cardinality at least $\mathfrak{c}$ using a construction similar to that of the Cantor set in $\mathbb{R}$.
\end{proof}
\end{prop}

Before we prove this section's main result, we prove its two special cases.

\begin{lem}
\label{lem:Kr4}
Let $G$ be a \kc\ Lie group and let $\varphi:G\to S$ be a continuous bijective group homomorphism onto a group $S$ which is the product of compact connected Lie groups. Then $G$ is compact.
\begin{proof}
If $S$ is a finite product of Lie groups, then it is a Lie group and the statement follows from Lemma \ref{lem:Kr1}. Assume that $S$ is an infinite product of non-trivial Lie groups. Then we can pick a finite number of them such that the sum of their dimensions is bigger than the dimension of $G$. Let $S_1$ be the product of those groups. Then $S_1$ is a Lie group of dimension bigger than $G$. It is also a closed subgroup of $S$, so $D=\varphi^{-1}(S_1)$ is a closed subgroup of $G$ and a \kc\ group. By Lemma \ref{lem:Kr1} for $\varphi|_D:D\to S_1$ we have that $D$ is homeomorphic to $S_1$, so $D$ is a Lie group of dimension bigger than the dimension of $G$ and we get a contradiction.
\end{proof}
\end{lem}

\begin{lem}
\label{lem:Kr3}
Let $G$ be a \kc\ Lie group and let $\varphi:G\to H$ be a continuous bijective group homomorphism onto a compact group $H$. Then $G$ is compact as well.
\begin{proof}
Let us assume that $G$ is not compact. By Proposition \ref{prop:Obs1/2}, $[G:G_0]<\mathfrak{c}$ and thus $[H:\overline{\varphi(G_0)}]\leq[H:\varphi(G_0)]<\mathfrak{c}$. By Proposition \ref{prop:ObsF}, ${H}/{\overline{\varphi(G_0)}}$ is finite, so $\overline{\varphi(G_0)}$ is clopen and $\widetilde{G}:=\varphi^{-1}(\overline{\varphi(G_0)})$ is clopen as well. Then $\widetilde{G}$ has a finite index in $G$, so it cannot be compact. Furthermore, $S:=\overline{\varphi(G_0)}$ is a compact connected group. Let $Z$ be a subgroup of $S$ as in Theorem \ref{tm:Po}. Then $C:=\varphi^{-1}(Z)$ is totally disconnected. Moreover, $C$ is a closed subgroup of $G$, so it is a \kc\ Lie group and its identity component is open, so it is discrete, which means that it is of cardinality less than $\mathfrak{c}$. Then by Proposition \ref{prop:ObsF} $Z$ is finite and so is $C$ and ${\widetilde{G}}/{C}$ is not compact. Finally by applying Lemma \ref{lem:Kr4} to $\psi:{\widetilde{G}}/{C}\ni[g]\mapsto[\varphi(g)]\in{S}/{Z}$ we get that ${\widetilde{G}}/{C}$ is compact and thus $\widetilde{G}$ is compact which leads to a contradiction.
\end{proof} 
\end{lem}

We are now ready to prove the Open Mapping Theorem with compact target group.

\begin{lem}
\label{lem:Kr2}
Let $G$ be a locally compact \kc\ group and let $\varphi:G\to H$ be a continuous bijective group homomorphism onto a compact group $H$. Then $G$ is compact as well.
\begin{proof}
Assume that $G$ is not compact. Let $W$ be an open almost connected subgroup of $G$ (as in Proposition \ref{prop:VD}). Then let $\widetilde{G}=\varphi^{-1}(\overline{\varphi(W)})$ and by arguing as in Lemma \ref{lem:Kr3} (i.e. applying Proposition \ref{prop:ObsF} to the space $H/\overline{\varphi(W)}$ of the left cosets) we get that $\widetilde{G}$ is a clopen subgroup of $G$ with finite index and in particular it is not compact. Let $K$ be a normal compact subgroup of $W$ such that ${W}/{K}$ is a Lie group (as in Theorem \ref{tm:GY}). We get that $\varphi(K)\lhd\varphi(W)$ and since $\varphi(K)$ is closed in $\overline{\varphi(W)}$ we get that $\varphi(K)\lhd\overline{\varphi(W)}$ and $K\lhd \widetilde{G}$. Let $S:={\overline{\varphi(W)}}/{\varphi(K)}$. Then $S$ is compact and for the canonical projection $\pi:\overline{\varphi(W)}\to S$ we get that $\pi\circ\varphi|_{\widetilde{G}}/:\widetilde{G}\to S$ is a continuous group epimorphism with kernel $K$ and thus by Proposition \ref{prop:Kr0} we construct a continuous bijective group homomorphism $\psi:{\widetilde{G}}/{K}\to S$. We observe that ${\widetilde{G}}/{K}$ is a locally compact \kc\ NSS group that is not compact and thus by Lemma \ref{lem:Kr3} we get a contradiction.
\end{proof}
\end{lem}

\section{Proof of the Open Mapping Theorem}

In this section we prove Theorem \ref{Twr1}, first in the case where $H$ is almost connected and then in the general case.

\begin{lem}
\label{lem:LL}
Let $G$, $H$ be locally compact groups such that $G$ is \kc\ and $H$ is almost connected and let $\varphi:G\to H$ be a continuous bijective group homomorphism. Then $\varphi$ is an isomorphism of topological groups.
\begin{proof}
Let $K$ be a compact normal subgroup of $H$ such that $S={H}/{K}$ is a Lie group (as in Theorem \ref{tm:GY}). Let $L=\varphi^{-1}(K)$. Then $L$ is a closed normal subgroup of $G$ and is \kc. From Lemma \ref{lem:Kr2} for $\varphi|_L:L\to K$ we get that $L$ is compact. Take any net $\{x_\sigma\}_{\sigma\in\Sigma}$ in $G$ such that $\varphi(x_\sigma)$ converges to the identity in $H$. To prove that $\varphi$ is open we want to check that $x_\sigma$ converges to the identity in $G$. It is enough to show that some subnet of $\{x_\sigma\}_{\sigma\in\Sigma}$ converges to the identity in $G$. For $\pi:H\to S$, $\pi\circ\varphi:G\to S$ is open from Lemma \ref{lem:Kr1}. Moreover $\pi\circ\varphi(x_\sigma)$ converges to the identity in $S$. Let $U$ be an identity neighbourhood in $G$ such that $\overline{U}$ is compact. We observe that $\pi\circ\varphi(U)$ is an identity neighbourhood in $S$, so $x_\sigma\in U\cdot \ker\pi\circ\varphi=U\cdot L \subseteq \overline{U}\cdot L$ for large enough $\sigma$ and thus $\{x_\sigma\}_{\sigma\in\Sigma}$ has a subnet $\{y_\lambda\}_{\lambda\in\Lambda}$ convergent in $\overline{U}\cdot L$ and since $\varphi(x_\sigma)$ converges to the identity in $H$ and $\varphi$ is a group monomorphism, $y_\lambda$ converges to the identity in $G$.
\end{proof}
\end{lem}

\proof[Proof of Theorem \ref{Twr1}] Let $W$ be an open almost connected open subgroup of $H$ (as in Proposition \ref{prop:VD}). Then $\varphi^{-1}(W)$ is an open subgroup of $G$, so $\varphi$ is open if and only if $\varphi|_{\varphi^{-1}(W)}$ is open and the latter is true by Lemma \ref{lem:LL}.

As a sidenote, it is worth mentioning that Theorem \ref{Twr1} is easily equivalent to the Closed Graph Theorem for locally compact \kc\ groups which reads as follows:

\begin{tm}
\label{thm:CGT}
Let $G$, $H$ be locally compact \kc\ groups and let $\varphi:G\to H$ be a group homomorphism. Then $\varphi$ is continuous if and only if its graph is closed in $G\times H$.
\end{tm}
We end the article by posing the following questions:

\textit{Can Theorem \ref{Twr0} be proven without the use of structural results for locally compact groups?}

\textit{Is there a result in spirit of Statement \ref{tm:Bairec} that can be proven in ZFC?}

\end{document}